\documentclass[twoside, journal, onecolumn, double, 11pt]{IEEEtran}

\usepackage{epsfig}
\usepackage{graphicx}
\usepackage{amsmath}
\usepackage{amssymb}
\usepackage{float}
\usepackage{url}
\newtheorem{definition}{Definition}
\newtheorem{lemma}{Lemma}

\newtheorem{corollary}{Corollary}
\newtheorem{theorem}{Theorem}

\newtheorem{remark}{Remark}
\newtheorem{example}{Example}

\newcommand{\naturals}{\ensuremath{\mathbb{N}}}
\newcommand{\reals}{\ensuremath{\mathbb{R}}}

\newcommand{\pr}{\ensuremath{\mathbb{P}}}
\newcommand{\expectation}{\ensuremath{\mathbb{E}}}

\begin{document}

\title{ \huge{Tightened Exponential Bounds for Discrete Time,
Conditionally Symmetric Martingales with Bounded Jumps}}

\markboth{To appear in the Statistics and Probability Letters. 
Final version: May 1, 2013.}{I. SASON: Tightened Exponential 
Bounds for Conditionally Symmetric Martingales with Bounded Jumps}

\author{\vspace*{0.2cm} \IEEEauthorblockN{Igal Sason\\
Department of Electrical Engineering\\
Technion - Israel Institute of Technology\\
Haifa 32000, Israel\\
\hspace*{-0.4cm} E-mail: sason@ee.technion.ac.il}}

\maketitle

\begin{abstract}
This letter derives some new exponential bounds for
discrete time, real valued, conditionally symmetric
martingales with bounded jumps. The new bounds are
extended to conditionally symmetric sub/ supermartingales,
and they are compared to some existing bounds.
\end{abstract}



\section{Introduction and Main Results}
\label{section: Introduction}
Classes of exponential bounds for
discrete-time real-valued martingales have been
extensively studied in the literature (see, e.g., Alon and Spencer (2008),
Azuma (1967), Burkholder (1991), Chung and Lu (2006), Dembo and Zeitouni (1997),
Dzhaparide and van Zanten (2001), Freedman (1975), Grama and E. Haeusler (2000),
Hoeffding (1963),
McDiarmid (1989, 1998), de la Pe\~{n}a (1999), de la Pe\~{n}a, Klass and Lai (2004),
Pinelis (1994) and Steiger (1969)).
This letter further assumes conditional symmetry of these martingales,
as is defined in the following.

\vspace*{0.1cm}
\begin{definition}
Let $\{X_k, \mathcal{F}_k\}_{k \in \naturals_0}$, where $ \naturals_0
\triangleq \naturals \cup \{0\}$, be a discrete-time
and real-valued martingale,
and let $\xi_k \triangleq X_k - X_{k-1}$ for every $k \in \naturals$
designate the jumps of the martingale.
Then $\{X_k, \mathcal{F}_k\}_{k \in \naturals_0}$
is called a {\em conditionally symmetric martingale} if, conditioned on
$\mathcal{F}_{k-1}$, the random variable $\xi_k$ is
symmetrically distributed around zero.
\label{definition: conditionally symmetric martingales}
\end{definition}

\vspace*{0.1cm}
Our goal in this letter is to demonstrate how the assumption of
the conditional symmetry improves the existing exponential inequalities for
discrete-time real-valued martingales with bounded increments.
Earlier results, serving as motivation, appear in
Section~4 of Dzhaparide and J. H. van Zanten (2001)
and Section~6 of de la Pe\~{n}a (1999).
The new exponential bounds are also extended to conditionally
symmetric submartingales or supermartingales, where the construction of
these objects is exemplified later in this section.
The relation of some of the exponential bounds derived in
this work with some existing bounds is discussed later in this letter.
Additional results addressing weak-type inequalities, maximal inequalities
and ratio inequalities for conditionally symmetric martingales
were derived in Os\c{e}kowski (2010a,b) and Wang (1991).

\subsection{Main Results}
Our main results for conditionally symmetric martingales
with bounded jumps are introduced in
Theorems~\ref{theorem: a concentration inequality for a symmetric
conditional distribution}, \ref{theorem: second inequality for
conditionally symmetric martingales} and~\ref{theorem: improved
inequality for the exercise of Amir and Ofer's book w.r.t.
conditionally symmetric martingales}.
Theorems~\ref{theorem: first refined concentration inequality}
and~\ref{theorem: from the book of Dembo and Zeitouni}
are existing bounds, for general martingales without the
conditional symmetry assumption, that are introduced in connection
to the new theorems. Corollaries~\ref{corollary: concentration
inequality for conditionally symmetric sub and supermartingales}
and~\ref{corollary: adaptation of Theorems 6 and 7 to supermartingales}
provide an extension of the new results to conditionally symmetric
sub/ supermartingales with bounded jumps.
Our first result is the following theorem.
\vspace*{0.1cm}
\begin{theorem}
Let $\{X_k, \mathcal{F}_k\}_{k \in \naturals_0}$ be a discrete-time
real-valued and conditionally symmetric martingale.
Assume that, for some fixed numbers $d, \sigma > 0$, the
following two requirements are satisfied a.s.:
\begin{equation}
| X_k - X_{k-1} | \leq d, \quad \quad
\text{Var} (X_k | \mathcal{F}_{k-1}) = \expectation \bigl[(X_k
- X_{k-1})^2 \, | \, \mathcal{F}_{k-1} \bigr] \leq \sigma^2
\label{eq: assumptions related to Theorem 2}
\end{equation}
for every $k \in \naturals$. Then, for every
$\alpha \geq 0$ and $n \in \naturals$,
\begin{equation}
\pr\left(\max_{1 \leq k \leq n} |X_k-X_0| \geq \alpha n\right)
\leq 2 \exp\bigl(-n E(\gamma, \delta) \bigr)
\label{eq: a concentration inequality for a symmetric
conditional distribution}
\end{equation}
where
\begin{equation}
\gamma \triangleq \frac{\sigma^2}{d^2}, \quad \delta \triangleq
\frac{\alpha}{d}  \label{eq: notation}
\end{equation}
and for $\gamma \in (0,1]$ and $\delta \in [0,1)$
\begin{eqnarray}
&& E(\gamma, \delta) \triangleq \delta x - \ln \Bigl( 1 + \gamma
\bigl[\cosh(x)-1 \bigr] \Bigr)
\label{eq: exponent for conditional symmetric martingales}\\[0.1cm]
&& x \triangleq \ln \left( \frac{\delta (1-\gamma) +
\sqrt{\delta^2 (1-\gamma)^2 + \gamma^2 (1-\delta^2)}}{\gamma (1-\delta)} \right).
\label{eq: optimal value of x in the exponent for conditional symmetric martingales}
\end{eqnarray}
If $\delta > 1$, then the probability on the left-hand side of
\eqref{eq: a concentration inequality for a symmetric
conditional distribution} is zero (so $E(\gamma, \delta) \triangleq +\infty$),
and $E(\gamma, 1) = \ln\bigl(\frac{2}{\gamma}\bigr).$ Furthermore, the exponent
$E(\gamma, \delta)$
is asymptotically optimal in the sense that there exists a conditionally
symmetric martingale, satisfying the conditions in
\eqref{eq: assumptions related to Theorem 2} a.s., that attains this
exponent in the limit where $n \rightarrow \infty$.
\label{theorem: a concentration inequality for a symmetric conditional distribution}
\end{theorem}

\vspace*{0.1cm}
\begin{remark}
From the above conditions, without any loss of generality,
$\sigma^2 \leq d^2$ and therefore $\gamma \in (0,1]$. This implies that
Theorem~\ref{theorem: a concentration inequality for a symmetric conditional distribution}
characterizes the exponent $E(\gamma, \delta)$ for all
values of $\gamma$ and $\delta$.
\end{remark}

\begin{corollary}
Let $\{U_k\}_{k=1}^{\infty} \in L^2(\Omega, \mathcal{F}, \pr)$ be i.i.d. and bounded
random variables with a symmetric distribution around their mean value.
Assume that $|U_1 - \expectation[U_1]| \leq d$ a.s. for some $d>0$, and $\text{Var}(U_1) \leq \gamma d^2$
for some $\gamma \in [0,1]$. Let $\{S_n\}$ designate the sequence of partial sums, i.e.,
$S_n \triangleq \sum_{k=1}^n U_k$ for every $n \in \naturals$. Then, for every $\alpha \geq 0$,
\begin{equation}
\pr\left(\max_{1 \leq k \leq n} \bigl|S_k - k \, \expectation(U_1) \bigr| \geq \alpha n \right)
\leq 2 \exp\bigl(-n E(\gamma, \delta) \bigr), \quad \forall \, n \in \naturals
\end{equation}
where $\delta \triangleq \frac{\alpha}{d}$, and $E(\gamma, \delta)$ is introduced
in \eqref{eq: exponent for conditional symmetric martingales}
and \eqref{eq: optimal value of x in the exponent for conditional symmetric martingales}.
\label{corollary: a tail bound for sums of i.i.d. and bounded RVs witha symmetric distribution}
\end{corollary}

\vspace*{0.1cm}
Theorem~\ref{theorem: a concentration inequality for a symmetric conditional distribution}
should be compared to the statement in
Theorem~6.1 of McDiarmid (1989)
(see also Corollary~2.4.7 in Dembo and Zeitouni (1997)),
which does not require the conditional symmetry property.
It gives the following result.
\vspace*{0.1cm}
\begin{theorem}
Let $\{X_k, \mathcal{F}_k\}_{k \in \naturals_0}$ be a
discrete-time real-valued martingale with bounded jumps. Assume that
the two conditions in \eqref{eq: assumptions related to Theorem 2}
are satisfied a.s. for every $k \in \naturals$. Then, for every
$\alpha \geq 0$ and $n \in \naturals$,
\begin{equation}
\pr\left(\max_{1 \leq k \leq n} |X_k-X_0| \geq \alpha n\right) \leq 2 \exp\left(-n
\, D\biggl(\frac{\delta+\gamma}{1+\gamma} \Big|\Big|
\frac{\gamma}{1+\gamma}\biggr) \right) \label{eq: first refined
concentration inequality}
\end{equation}
where $\gamma$ and $\delta$ are introduced in \eqref{eq: notation},
and
\begin{equation}
D(p || q) \triangleq p \ln\Bigl(\frac{p}{q}\Bigr) + (1-p)
\ln\Bigl(\frac{1-p}{1-q}\Bigr), \quad \forall \, p, q \in [0,1]
\label{eq: divergence}
\end{equation}
is the divergence (also known as relative entropy or
Kullback-Leibler distance) to the natural base between
the two probability distributions $(p,1-p)$ and $(q,1-q)$ (with
the convention that $0 \log 0$ is zero in the case where
$p$ or $q$ take on the values zero or one).
If $\delta>1$, then the probability on the left-hand
side of \eqref{eq: first refined concentration inequality}
is zero. Furthermore, the exponent on the right-hand side of
\eqref{eq: first refined concentration inequality} is
asymptotically optimal under the assumptions of this theorem.
\label{theorem: first refined concentration inequality}
\end{theorem}

\begin{remark}
The two exponents in
Theorems~\ref{theorem: a concentration inequality for a symmetric conditional distribution}
and~\ref{theorem: first refined concentration inequality}
are both discontinuous at $\delta=1$. This is consistent with the assumption of the bounded jumps
that implies that $\pr(|X_n - X_0| \geq n d \delta)$ is equal to zero if $\delta > 1$.

If $\delta \rightarrow 1^-$ then, from
\eqref{eq: exponent for conditional symmetric martingales} and
\eqref{eq: optimal value of x in the exponent for conditional symmetric martingales},
for every $\gamma \in (0,1]$,
\begin{eqnarray}
\lim_{\delta \rightarrow 1^-} E(\gamma, \delta)
= \lim_{x \rightarrow \infty} \left[x - \ln\bigl(1+\gamma (\cosh(x)-1)\bigr)\right]
= \ln\left(\frac{2}{\gamma}\right).
\label{eq: left limit of the exponent in Theorem 1 at delta=1}
\end{eqnarray}
On the other hand, the right limit at $\delta=1$ is infinity since
$E(\gamma, \delta) = +\infty$ for every $\delta>1$.
The same discontinuity also exists for the
exponent in Theorem~\ref{theorem: first refined concentration inequality}
where the right limit at $\delta=1$ is infinity, and the left limit is equal to
\begin{equation}
\lim_{\delta \rightarrow 1^-} D\biggl(\frac{\delta+\gamma}{1+\gamma} \Big|\Big|
\frac{\gamma}{1+\gamma}\biggr) = \ln\left(1+\frac{1}{\gamma}\right)
\label{eq: left limit of the exponent in Theorem 2 at delta=1}
\end{equation}
where the last equality follows from \eqref{eq: divergence}.
A comparison of the limits in \eqref{eq: left limit of the exponent in Theorem 1 at delta=1}
and \eqref{eq: left limit of the exponent in Theorem 2 at delta=1} is consistent with
the improvement that is obtained in
Theorem~\ref{theorem: a concentration inequality for a symmetric conditional distribution}
as compared to Theorem~\ref{theorem: first refined concentration inequality} due to the
additional assumption of the conditional symmetry that is relevant if $\gamma \in (0,1)$.
It can be verified that the two exponents coincide if $\gamma=1$
(which is equivalent to removing the constraint on the conditional variance), and their common value is equal to
\begin{equation}
f(\delta) = \left\{
\begin{array}{ll}
\ln(2) \Bigl[1 - h_2\left(\frac{1-\delta}{2} \right) \Bigr], \quad &0
\leq \delta \leq 1
\\[0.1cm]
+\infty, \quad & \delta > 1
\end{array}
\right. \label{eq: f}
\end{equation}
where $h_2(x) \triangleq -x \log_2(x) - (1-x) \log_2(1-x)$ for $0
\leq x \leq 1$ denotes the binary entropy function to the base~2 (with
the convention that it is defined to be zero at $x=0$ or $x=1$).
\end{remark}

\vspace*{0.1cm}
Theorem~\ref{theorem: a concentration inequality for
a symmetric conditional distribution} provides an
improvement over the bound in
Theorem~\ref{theorem: first refined concentration inequality}
for conditionally symmetric martingales with
bounded jumps. The bounds in
Theorems~\ref{theorem: a concentration inequality for
a symmetric conditional distribution}
and~\ref{theorem: first refined concentration inequality}
depend on the conditional variance of the martingale, but
they do not take into consideration conditional moments
of higher orders. The following bound generalizes the
bound in Theorem~\ref{theorem: a concentration inequality for a symmetric conditional distribution},
but it does not admit in general a closed-form expression.
\vspace*{0.1cm}
\begin{theorem}
Let $\{X_k, \mathcal{F}_k\}_{k \in \naturals_0}$ be a
discrete-time and real-valued conditionally symmetric martingale. Let
$m \in \naturals$ be an even number, and
assume that the following conditions hold a.s. for every
$k \in \naturals$
\begin{eqnarray*}
| X_k - X_{k-1} | \leq d, \nonumber \quad \quad
\expectation \bigl[(X_k - X_{k-1})^l \, | \,
\mathcal{F}_{k-1} \bigr] \leq \mu_l, \; \;
\forall \, l \in \{2, 4, \ldots, m\} \label{eq: gamma_l}
\end{eqnarray*}
for some $d>0$ and non-negative numbers
$\{\mu_2, \mu_4, \ldots, \mu_m\}$. Then,
for every $\alpha \geq 0$ and $n \in \naturals$,
\begin{eqnarray}
\pr\left(\max_{1 \leq k \leq n} |X_k-X_0| \geq \alpha n\right)
\leq 2 \left\{\min_{x \geq 0}
\, e^{-\delta x} \left[1 + \sum_{l=1}^{\frac{m}{2}-1}
\frac{(\gamma_{2l}-\gamma_m) \, x^{2l}}{(2l)!} +
\gamma_m \bigl(\cosh(x)-1\bigr) \right] \right\}^n
\label{eq: inequality in the 2nd approach}
\end{eqnarray}
where
\begin{equation}
\delta \triangleq \frac{\alpha}{d}, \quad
\gamma_{2l} \triangleq \frac{\mu_{2l}}{d^{2l}}, \; \;
\forall \, l \in \Bigl\{1, \ldots, \frac{m}{2}\Bigr\}.
\label{eq: gamma and delta for the second theorem
on conditionally symmetric martingales}
\end{equation}
\label{theorem: second inequality for conditionally
symmetric martingales}
\end{theorem}

We consider in the following a different type of exponential inequalities for
conditionally symmetric martingales with bounded jumps.
\vspace*{0.1cm}
\begin{theorem}
Let $\{X_n, \mathcal{F}_n\}_{n \in \naturals_0}$ be a
discrete-time real-valued and conditionally symmetric
martingale. Assume that
there exists a fixed number $d>0$ such that
$\xi_k \triangleq X_k - X_{k-1} \leq d$ a.s. for every
$k \in \naturals$. Let
\begin{equation}
Q_n \triangleq \sum_{k=1}^n \expectation[\xi_k^2 \, | \,
\mathcal{F}_{k-1}] \label{eq: Q}
\end{equation}
with $Q_0 \triangleq 0$, be the predictable quadratic variation of the martingale
up to time $n$. Then, for every $z,r>0$,
\begin{equation}
\pr\left( \max_{1 \leq k \leq n} (X_k - X_0)
\geq z, \, Q_n \leq r \; \; \text{for some $n \in \naturals$}\right)
\leq \exp \left(-\frac{z^2}{2r} \cdot C\left(\frac{zd}{r} \right) \right)
\label{eq: improved inequality for the exercise of Amir and
Ofer's book w.r.t. conditionally symmetric martingales}
\end{equation}
where
\begin{equation}
C(u) \triangleq \frac{2[u \sinh^{-1}(u) - \sqrt{1+u^2}+1]}{u^2} \, , \quad \forall
\, u>0. \label{eq: C}
\end{equation}
\label{theorem: improved inequality for the exercise of Amir and Ofer's
book w.r.t. conditionally symmetric martingales}
\end{theorem}

Theorem~\ref{theorem: improved inequality for the exercise of Amir and Ofer's
book w.r.t. conditionally symmetric martingales} should be compared to
Theorem~1.6 in Freedman (1975)
(see also Exercise~2.4.21(b) in Dembo and Zeitouni (1997)) that was stated
without the requirement for the conditional symmetry of the martingale.
It provides the following result:
\vspace*{0.1cm}
\begin{theorem}
Let $\{X_n, \mathcal{F}_n\}_{n \in \naturals_0}$ be a
discrete-time real-valued martingale. Assume that
there exists a fixed number $d>0$ such that
$\xi_k \triangleq X_k - X_{k-1} \leq d$ a.s. for every
$k \in \naturals$. Then, for every $z,r>0$,
\begin{eqnarray}
\pr\left( \max_{1 \leq k \leq n} (X_k - X_0)
\geq z, \, Q_n \leq r \; \; \text{for some $n \in \naturals$}\right)
\leq \exp \left(-\frac{z^2}{2r} \cdot
B\left(\frac{zd}{r} \right) \right) \label{eq: inequality in
exercise of Amir and Ofer's book}
\end{eqnarray}
where
\begin{equation}
B(u) \triangleq \frac{2[(1+u) \ln(1+u) - u]}{u^2} \, , \quad \forall
\, u>0. \label{eq: B}
\end{equation}
\label{theorem: from the book of Dembo and Zeitouni}
\end{theorem}

The proof of Theorem~1.6 in Freedman (1975) is modified
by using Bennett's inequality (see Bennett (1962)) for the derivation of
the original bound in Theorem~\ref{theorem: from the book of Dembo and Zeitouni}
(without the conditional symmetry requirement). Furthermore, this modified proof
serves to derive the improved bound in
Theorem~\ref{theorem: improved inequality for the exercise
of Amir and Ofer's book w.r.t. conditionally symmetric martingales}
under the conditional symmetry assumption.

In the following, the inequalities are extended to discrete-time, real-valued, and
conditionally symmetric sub/ supermartingales.

\vspace*{0.1cm}
\begin{definition}
Let $\{X_k, \mathcal{F}_k\}_{k \in \naturals_0}$ be a discrete-time
real-valued sub or supermartingale,
and $\eta_k \triangleq X_k - \expectation[X_k | \mathcal{F}_{k-1}]$
for every $k \in \naturals.$
Then the martingale $\{X_k, \mathcal{F}_k\}_{k \in \naturals_0}$
is called, respectively, a conditionally symmetric sub or supermartingale
if, conditioned on $\mathcal{F}_{k-1}$, the random variable $\eta_k$ is
symmetrically distributed around zero.
\label{definition: conditionally symmetric sub/supermartingales}
\end{definition}
\begin{remark}
For martingales, $\eta_k = \xi_k$ for every $k \in \naturals$, so we obtain
consistency with Definition~\ref{definition: conditionally symmetric martingales}.
\end{remark}

\vspace*{0.1cm}
An extension of Theorem~\ref{theorem: a concentration inequality
for a symmetric conditional distribution} to conditionally symmetric
sub and supermartingales is introduced in the following.

\begin{corollary}
Let $\{X_k, \mathcal{F}_k\}_{k \in \naturals_0}$ be a
discrete-time, real-valued and conditionally symmetric supermartingale.
Assume that, for some constants $d, \sigma > 0$, the following two
requirements are satisfied a.s.
\begin{equation}
\eta_k \leq d, \quad \quad
\text{Var} (X_k | \mathcal{F}_{k-1}) \triangleq \expectation
\bigl[\eta_k^2 \, | \, \mathcal{F}_{k-1} \bigr] \leq
\sigma^2 \label{eq: condition on the conditional variance}
\end{equation}
for every $k \in \naturals$. Then, for every $\alpha \geq
0$ and $n \in \naturals$,
\begin{equation}
\pr\Bigl(\max_{1 \leq k \leq n} (X_k-X_0) \geq \alpha n\Bigr) \leq \exp\bigl(-n \,
E(\gamma, \delta) \bigr) \label{eq: concentration inequality
for conditionally symmetric supermartingales}
\end{equation}
where $\gamma$ and $\delta$ are defined in \eqref{eq:
notation}, and $E(\gamma, \delta)$ is introduced in
\eqref{eq: exponent for conditional symmetric martingales}.
Alternatively, if $\{X_k,
\mathcal{F}_k\}_{k \in \naturals_0}$ is a conditionally
symmetric submartingale, the same
bound holds for $\pr\bigl(\min_{1 \leq k \leq n} (X_k-X_0) \leq
-\alpha n\bigr)$ provided that
$\eta_k \geq -d$ and the second condition in
\eqref{eq: condition on the conditional variance}
hold a.s. for every $k \in \naturals$.
If $\delta>1$, then these two probabilities are
zero. \label{corollary: concentration inequality for conditionally
symmetric sub and supermartingales}
\end{corollary}

\vspace*{0.1cm}
The following statement extends Theorem~\ref{theorem: improved
inequality for the exercise of Amir and Ofer's book w.r.t.
conditionally symmetric martingales} to conditionally symmetric
supermartingales.

\begin{corollary}
Let $\{X_n, \mathcal{F}_n\}_{n \in \naturals_0}$ be a
discrete-time, real-valued supermartingale. Assume that
there exists a fixed number $d>0$ such that $\eta_k \leq d$
a.s. for every $k \in \naturals$. Let $\{Q_n\}_{n \in \naturals_0}$ be
the predictable quadratic variations of the supermartingale, i.e.,
$Q_n \triangleq \sum_{k=1}^n \expectation[\eta_k^2 \, | \, \mathcal{F}_{k-1}]$
for every $n \in \naturals$ with $Q_0 \triangleq 0$.
Then, the result in \eqref{eq: inequality in
exercise of Amir and Ofer's book} holds.
Furthermore, if the supermartingale is conditionally symmetric,
then the improved bound in \eqref{eq: improved inequality for the exercise
of Amir and Ofer's book w.r.t. conditionally symmetric martingales}
holds.
\label{corollary: adaptation of Theorems 6 and 7 to supermartingales}
\end{corollary}

\subsection{Construction of Discrete-Time, Real-Valued and Conditionally Symmetric Sub/ Supermartingales}
\label{subsection: Construction of conditionally symmetric sub/ supermartingales}
Before proving the tightened inequalities for discrete-time
conditionally symmetric sub/ supermartingales, it is worth
exemplifying the construction of these objects.
\begin{example}
Let $(\Omega, \mathcal{F}, \pr)$ be a probability space, and let
$\{U_k\}_{k \in \naturals} \subseteq L^1(\Omega, \mathcal{F},
\pr)$ be a sequence of independent random variables with zero mean.
Let $\{\mathcal{F}_k\}_{k \geq 0}$ be the natural filtration of sub
$\sigma$-algebras of $\mathcal{F}$, where
$\mathcal{F}_0 = \{\emptyset, \Omega\}$ and $\mathcal{F}_k =
\sigma(U_1, \ldots, U_k)$ for $k \geq 1.$ Furthermore,
for $k \in \naturals$, let $A_k \in L^{\infty}(\Omega, \mathcal{F}_{k-1},
\pr)$ be an $\mathcal{F}_{k-1}$-measurable random variable with a
finite essential supremum. Define a new sequence of random variables
in $L^1(\Omega, \mathcal{F}, \pr)$ where
$$X_n = \sum_{k=1}^n A_k U_k, \; \; \forall \, n \in \naturals$$
and $X_0 = 0$.
Then, $\{X_n, \mathcal{F}_n\}_{n \in \naturals_0}$ is a martingale.
Let us assume that the random variables $\{U_k\}_{k \in
\naturals}$ are symmetrically distributed around zero. Note that
$X_n = X_{n-1} + A_n U_n$ where $A_n$
is $\mathcal{F}_{n-1}$-measurable and $U_n$ is independent of the
$\sigma$-algebra $\mathcal{F}_{n-1}$ (due to the independence of
the random variables $U_1, \ldots, U_n$). It therefore follows
that for every $n \in \naturals$, given $\mathcal{F}_{n-1}$, the
random variable $X_n$ is symmetrically distributed around its
conditional expectation $X_{n-1}$. Hence, the martingale $\{X_n,
\mathcal{F}_n\}_{n \in \naturals_0}$ is conditionally
symmetric.
\label{example1: constructing a conditionally symmetric martingale from iid RVs}
\end{example}

\vspace*{0.1cm}
\begin{example}
In continuation of
Example~\ref{example1: constructing a conditionally symmetric martingale from iid RVs},
let $\{X_n, \mathcal{F}_n\}_{n \in \naturals_0}$ be a martingale, and
define $Y_0=0$ and $$Y_n = \sum_{k=1}^n A_k (X_k - X_{k-1}), \quad
\forall \, n \in \naturals.$$ The sequence $\{Y_n,
\mathcal{F}_n\}_{n \in \naturals_0}$ is a martingale.
If $\{X_n, \mathcal{F}_n\}_{n \in \naturals_0}$ is a conditionally
symmetric martingale then also the
martingale $\{Y_n, \mathcal{F}_n\}_{n \in \naturals_0}$
is conditionally symmetric (since $Y_n =
Y_{n-1} + A_n (X_n - X_{n-1})$, and by assumption $A_n$ is
$\mathcal{F}_{n-1}$-measurable).
\label{example2: constructing a conditionally symmetric martingale from a conditionally symmetric martingale}
\end{example}

\vspace*{0.1cm}
\begin{example}
In continuation of Example~\ref{example1: constructing a conditionally
symmetric martingale from iid RVs}, let $\{U_k\}_{k \in \naturals}$
be independent random variables with a symmetric
distribution around their expected value, and also assume that
$\expectation(U_k) \leq 0$ for every $k \in \naturals$.
Furthermore, let $A_k \in L^{\infty}(\Omega, \mathcal{F}_{k-1},
\pr)$, and assume that a.s. $A_k \geq 0$
for every $k \in \naturals$. Let $\{X_n,
\mathcal{F}_n\}_{n \in \naturals_0}$ be a martingale as defined in
Example~\ref{example1: constructing a conditionally
symmetric martingale from iid RVs}.
Note that $X_n = X_{n-1} + A_n U_n$ where $A_n$ is
non-negative and $\mathcal{F}_{n-1}$-measurable, and
$U_n$ is independent of $\mathcal{F}_{n-1}$
and symmetrically distributed around its average.
This implies that $\{X_n, \mathcal{F}_n\}_{n \in \naturals_0}$
is a conditionally symmetric supermartingale.
\label{example3: constructing a conditionally symmetric sub
or supermartingale from iid RVs}
\end{example}

\vspace*{0.1cm}
\begin{example}
In continuation of
Examples~\ref{example2: constructing a conditionally symmetric
martingale from a conditionally symmetric martingale}
and~\ref{example3: constructing a conditionally symmetric sub
or supermartingale from iid RVs}, let
$\{X_n, \mathcal{F}_n\}_{n \in \naturals_0}$ be a conditionally
symmetric supermartingale. Define $\{Y_n\}_{n \in \naturals_0}$
as in Example~\ref{example2: constructing a conditionally symmetric
martingale from a conditionally symmetric martingale} where
$A_k$ is non-negative a.s. and $\mathcal{F}_{k-1}$-measurable
for every $k \in \naturals$.
Then $\{Y_n, \mathcal{F}_n\}_{n \in \naturals_0}$ is a conditionally
symmetric supermartingale.
\label{example4: constructing a conditionally symmetric sub/supermartingale
from a conditionally symmetric sub/supermartingale}
\end{example}


\section{Proofs}
\label{section: proofs}
\subsection{Proof of Theorem~\ref{theorem: a concentration inequality
for a symmetric conditional distribution}}
We rely here on the proof of the existing bound that is stated
in Theorem~\ref{theorem: first refined concentration inequality},
for discrete-time real-valued martingales with bounded jumps
(see Theorem~6.1 in McDiarmid (1989)
and Corollary~2.4.7 in Dembo and Zeitouni (1997)),
and then deviate from this proof at the point where the
additional property of the conditional symmetry of the
martingale is taken into consideration for the derivation
of the improved exponential inequality in Theorem~\ref{theorem:
a concentration inequality for a symmetric conditional distribution}.

Write
$X_n - X_0 = \sum_{k=1}^n \xi_k$ where $\xi_k \triangleq X_k - X_{k-1}$ for
$k \in \naturals$. Since
$\{X_k - X_0, \mathcal{F}_k\}_{k \in \naturals_0}$
is a martingale, $h(x) = \exp(tx)$ is a convex
function on $\reals$ for every $t \in \reals$, and a composition of a
convex function with a martingale gives a submartingale w.r.t.
the same filtration, $\bigl\{\exp(t(X_k - X_0)), \mathcal{F}_k\bigr\}_{k \in \naturals_0}$
is a sub-martingale for every $t \in \reals$. By applying the maximal inequality
for submartingales, then for every $\alpha \geq 0$ and $n \in \naturals$
\begin{eqnarray}
&& \pr\Bigl(\max_{1 \leq k \leq n} (X_k - X_0) \geq \alpha n\Bigr)  \nonumber \\
&& \leq \exp(-\alpha n t) \;
\expectation \Bigl[ \exp \bigl(t (X_n - X_0) \bigr) \Bigr] \quad \quad \forall \, t \geq 0 \nonumber \\
&& = \exp(-\alpha n t) \; \expectation \left[\exp \biggl(t
\sum_{k=1}^n \xi_k \biggr) \right].
\label{eq: maximal inequality for submartingales}
\end{eqnarray}
Furthermore,
\vspace*{-0.4cm}
\begin{equation}
\expectation \biggl[ \exp \biggl(t \sum_{k=1}^n \xi_k \biggr)
\biggr] = \expectation \Biggl[ \exp \biggl(t \sum_{k=1}^{n-1} \xi_k
\biggr) \, \expectation \bigl[ \exp(t \xi_n) \, | \,
\mathcal{F}_{n-1} \bigr] \Biggr] \label{eq: smoothing theorem}
\end{equation}
where this equality holds since $\exp \bigl(t
\sum_{k=1}^{n-1} \xi_k \bigr)$ is $\mathcal{F}_{n-1}$-measurable.

\vspace*{0.1cm}
In order to prove Theorem~\ref{theorem: a concentration inequality
for a symmetric conditional distribution} for a discrete-time, real-valued and
conditionally symmetric martingale with bounded jumps, we deviate from the proof
of Theorem~\ref{theorem: first refined concentration inequality}. This is done
by a replacement of Bennett's inequality (see Bennett (1962)) for the conditional
expectation with a tightened bound under the conditional symmetry assumption. The
following lemma appears in several probability textbooks on stochastic ordering
(see, e.g., Denuit et al. (2005)), and will be useful in our analysis.

\begin{lemma}
Let $X$ be a real-valued random variable with a symmetric distribution around zero,
and a support $[-d, d]$, and assume that
$\expectation[X^2] = \text{Var}(X) \leq \gamma d^2$ for some $d>0$ and
$\gamma \in [0,1]$.
Let $h$ be a real-valued convex function, and assume that $h(d^2) \geq h(0)$. Then
\begin{equation}
\expectation[h(X^2)] \leq (1-\gamma) h(0) + \gamma h(d^2)
\label{eq: inequality for symmetric distributions}
\end{equation}
where equality holds for the
symmetric distribution
\begin{eqnarray}
\pr(X = d) = \pr(X = -d) = \frac{\gamma}{2}, \quad
\pr(X = 0) = 1-\gamma.
\label{eq: symmetric distribution}
\end{eqnarray}
\label{lemma: inequality for symmetric distributions}
\end{lemma}
\begin{proof}
Since $h$ is convex and $\text{supp}(X)=[-d,d]$,
then a.s. $h(X^2) \leq h(0) + \left(\frac{X}{d}\right)^2
\bigl(h(d^2)-h(0)\bigr).$
Taking expectations on both sides gives \eqref{eq: inequality for symmetric distributions},
which holds with equality for the symmetric distribution in \eqref{eq: symmetric distribution}.
\end{proof}

\vspace*{0.1cm}
\begin{corollary}
If $X$ is a random variable that satisfies the three requirements in
Lemma~\ref{lemma: inequality for symmetric distributions} then, for
every $\lambda \in \reals$,
\begin{equation}
\expectation\bigl[\exp(\lambda X)\bigr] \leq 1 + \gamma
\bigl[\cosh(\lambda d)-1 \bigr] \label{eq:
Refined Bennett's inequality for symmetric distributions}
\end{equation}
and \eqref{eq: Refined Bennett's inequality for symmetric distributions}
holds with equality for the symmetric distribution in
Lemma~\ref{lemma: inequality for symmetric distributions},
independently of the value of $\lambda$.
\label{corollary: refined Bennett's inequality for
symmetric distributions}
\end{corollary}
\begin{proof}
For every $\lambda \in \reals$, due to the symmetric distribution
of $X$, $\expectation\bigl[\exp(\lambda X)\bigr] =
\expectation\bigl[\cosh(\lambda X)\bigr].$
The claim now follows from
Lemma~\ref{lemma: inequality for symmetric distributions}
since, for every $x \in \reals$, $\cosh(\lambda x) = h(x^2)$
where $h(x) \triangleq \sum_{n=0}^{\infty} \frac{\lambda^{2n}
|x|^n}{(2n)!}$ is a convex function ($h$ is convex since it is
a linear combination, with non-negative coefficients, of
convex functions), and $h(d^2) = \cosh(\lambda d) \geq 1 = h(0)$.
\end{proof}

\vspace*{0.1cm}
We continue with the proof of Theorem~\ref{theorem: a
concentration inequality for a symmetric conditional
distribution}. Under the assumption of this theorem, for every $k
\in \naturals$, the random variable $\xi_k \triangleq X_k -
X_{k-1}$ satisfies a.s.
$\expectation[\xi_k \, | \, \mathcal{F}_{k-1}] = 0$ and
$\expectation[(\xi_k)^2 \, | \, \mathcal{F}_{k-1}] \leq \sigma^2$.
Applying Corollary~\ref{corollary: refined Bennett's inequality for
symmetric distributions} for the conditional law of $\xi_k$ given
$\mathcal{F}_{k-1}$, it follows that for every $k \in \naturals$
and $t \in \reals$
\begin{equation}
\expectation \left[ \exp(t \xi_k) \, | \, \mathcal{F}_{k-1}
\right] \leq 1 + \gamma \bigl[ \cosh(td) - 1 \bigr]
\label{eq: Refined Bennett's inequality for the conditional law of
xi_k with symmetry condition}
\end{equation}
holds a.s., and therefore it follows from \eqref{eq: smoothing theorem} and
\eqref{eq: Refined Bennett's inequality for the conditional law of xi_k with
symmetry condition} that for every $t \in \reals$
\begin{equation}
\expectation \biggl[ \exp \biggl(t \sum_{k=1}^n \xi_k \biggr)
\biggr] \leq \Bigl( 1 + \gamma \bigl[ \cosh(td) - 1 \bigr] \Bigr)^n.
\label{eq: inequality for the expected value of the sum under a symmetry assumption}
\end{equation}
Therefore, from \eqref{eq: maximal inequality for submartingales}, for every $t \geq 0$,
\begin{eqnarray}
\pr\Bigl(\max_{1 \leq k \leq n} (X_k-X_0) \geq \alpha n\Bigr) \leq \exp(-\alpha n t)
\Bigl( 1 + \gamma \bigl[ \cosh(td) - 1 \bigr] \Bigr)^n.
\label{eq: intermediate step}
\end{eqnarray}
From \eqref{eq: notation}, and by using a replacement of $td$ with $x$, then
for an arbitrary $\alpha \geq 0$ and $n \in \naturals$
\begin{equation}
\pr\Bigl(\max_{1 \leq k \leq n} (X_k-X_0) \geq \alpha n\Bigr)
\leq  \inf_{x \geq 0} \; \left\{ \exp\left(-n \Bigl[\delta x -  \ln \bigl( 1 + \gamma
\bigl[ \cosh(x) - 1 \bigr] \bigr) \Bigr] \right) \right\}.
\label{eq: a one-sided concentration inequality for a symmetric conditional distribution}
\end{equation}
An optimization over the non-negative parameter $x$ gives the
solution for the optimized parameter in \eqref{eq: optimal
value of x in the exponent for conditional symmetric martingales}.
Applying \eqref{eq: a one-sided concentration inequality for a
symmetric conditional distribution} to the martingale $\{-X_k,
\mathcal{F}_k\}_{k \in \naturals_0}$ gives the same bound
on $\pr(\min_{1 \leq k \leq n} (X_k - X_0) \leq -\alpha n)$.
This completes the proof of
Theorem~\ref{theorem: a concentration inequality for a symmetric
conditional distribution}.

\vspace*{0.1cm}
{\em Proof for the asymptotic optimality of the exponents in
Theorems~\ref{theorem: a concentration inequality for a symmetric conditional distribution}
and~\ref{theorem: first refined concentration inequality}}:
In the following, we show that under the conditions of
Theorem~\ref{theorem: a concentration inequality for a symmetric conditional distribution},
the exponent $E(\gamma, \delta)$ in
\eqref{eq: exponent for conditional symmetric martingales} and
\eqref{eq: optimal value of x in the exponent for conditional symmetric martingales}
is asymptotically optimal. To show this,
let $d>0$ and $\gamma \in (0,1]$, and let $U_1, U_2, \ldots$ be i.i.d.
random variables whose probability distribution is given by
\begin{eqnarray}
\pr(U_i = d) = \pr(U_i = -d) = \frac{\gamma}{2}, \quad
\pr(U_i = 0) = 1-\gamma, \quad \forall \, i \in \naturals.
\label{eq: achievable symmetric distribution}
\end{eqnarray}
Consider the particular case of the conditionally symmetric martingale
$\{X_n, \mathcal{F}_n\}_{n \in \naturals_0}$ in
Example~\ref{example1: constructing a conditionally symmetric martingale from iid RVs}
(see Section~\ref{subsection: Construction of conditionally symmetric sub/ supermartingales})
where $X_n \triangleq \sum_{i=1}^n U_i$ for $n \in \naturals$, and $X_0 \triangleq 0$.
It follows that $|X_n - X_{n-1}| \leq d$ and $\text{Var}(X_n | \mathcal{F}_{n-1}) = \gamma d^2$
a.s. for every $n \in \naturals$.
From Cram\'{e}r's theorem in $\reals$, for every
$\alpha \geq \expectation[U_1] = 0$,
\begin{eqnarray}
&& \lim_{n \rightarrow \infty} \frac{1}{n} \,
\ln \pr(X_n - X_0 \geq \alpha n) \nonumber \\
&& = \lim_{n \rightarrow \infty} \frac{1}{n} \,
\ln \pr\biggl(\frac{1}{n} \sum_{i=1}^n U_i
\geq \alpha\biggr) \nonumber \\
&& = -I(\alpha) \label{eq: Cramer's theorem}
\end{eqnarray}
where the rate function is given by
\begin{equation}
I(\alpha) = \sup_{t \geq 0}
\left\{t \alpha - \ln \expectation[\exp(t U_1)] \right\}
\label{eq: rate function}
\end{equation}
(see, e.g., Theorem~2.2.3 in Dembo and Zeitouni (1997) and
Lemma~2.2.5(b) in Dembo and Zeitouni (1997) for the
restriction of the supremum to the interval $[0, \infty)$).
From \eqref{eq: achievable symmetric distribution}
and \eqref{eq: rate function}, for every $\alpha \geq 0$,
$$I(\alpha) = \sup_{t \geq 0} \left\{ t \alpha -
\ln \bigl(1+\gamma [\cosh(td)-1]\bigr) \right\}$$
but this is equivalent to the optimized exponent on the
right-hand side of \eqref{eq: intermediate step},
giving the exponent of the bound in
Theorem~\ref{theorem: a concentration inequality
for a symmetric conditional distribution}.
Hence, $I(\alpha) = E(\gamma, \delta)$ in
\eqref{eq: exponent for conditional symmetric martingales}
and \eqref{eq: optimal value of x in the exponent for
conditional symmetric martingales}. This proves that the
exponent of the bound in Theorem~\ref{theorem: a concentration
inequality for a symmetric conditional distribution} is
indeed asymptotically optimal in the sense that there
exists a discrete-time real-valued and conditionally
symmetric martingale, satisfying the conditions in
\eqref{eq: assumptions related to Theorem 2} a.s., that
attains this exponent in the limit where $n \rightarrow \infty$.
The proof for the asymptotic optimality of the exponent in
Theorem~\ref{theorem: first refined concentration inequality}
(see the right-hand side of \eqref{eq: first refined concentration inequality})
is similar to the proof for
Theorem~\ref{theorem: a concentration inequality for a symmetric conditional distribution},
except that the i.i.d. random variables $U_1, U_2, \ldots$
are now distributed as follows:
$$\pr(U_i = d) = \frac{\gamma}{1+\gamma}, \quad \pr(U_i = -\gamma d) = \frac{1}{1+\gamma},
\quad \forall \, i \in \naturals$$
and, as before, the martingale $\{X_n, \mathcal{F}_n\}_{n \in \naturals_0}$
is defined by $X_n = \sum_{i=1}^n U_i$ and $\mathcal{F}_n = \sigma(U_1, \ldots, U_n)$
for every $n \in \naturals$ with $X_0=0$ and $\mathcal{F}_0=\{\emptyset, \Omega\}$
(in this case, it is not a conditionally symmetric martingale unless $\gamma=1$).

\subsection{Proof of Theorem~\ref{theorem: second inequality for conditionally
symmetric martingales}}
The starting point of the proof of Theorem~\ref{theorem: second inequality for conditionally
symmetric martingales} relies on \eqref{eq: maximal inequality for submartingales}
and \eqref{eq: smoothing theorem}.
For every $k \in \naturals$ and $t \in \reals$, since
$\expectation\bigl[\xi_k^{2l-1} \, | \, \mathcal{F}_{k-1}\bigr] = 0$ for every
$l \in \naturals$ (due to the conditionally symmetry property of the martingale),
\begin{eqnarray}
&& \hspace*{-0.5cm} \expectation \bigl[ \exp(t \xi_k) |
\mathcal{F}_{k-1} \bigr] \nonumber \\
&& \hspace*{-0.5cm} = 1 + \sum_{l=1}^{\frac{m}{2}-1} \frac{t^{2l} \,
\expectation\bigl[\xi_k^{2l} \, | \, \mathcal{F}_{k-1}\bigr]}{(2l)!} \,
+ \sum_{l=\frac{m}{2}}^{\infty}
\frac{t^{2l} \, \expectation\bigl[\xi_k^{2l} \, | \, \mathcal{F}_{k-1}\bigr]}{(2l)!} \nonumber \\
&& \hspace*{-0.5cm} = 1 + \sum_{l=1}^{\frac{m}{2}-1} \frac{(td)^{2l} \,
\expectation\bigl[\bigl(\frac{\xi_k}{d}\bigr)^{2l} \, | \, \mathcal{F}_{k-1}\bigr]}{(2l)!} \,
+ \sum_{l=\frac{m}{2}}^{\infty}
\frac{(td)^{2l} \, \expectation\bigl[\bigl(\frac{\xi_k}{d}\bigr)^{2l} \, | \, \mathcal{F}_{k-1}\bigr]}{(2l)!} \nonumber \\
&& \hspace*{-0.5cm} \leq 1 + \sum_{l=1}^{\frac{m}{2}-1} \frac{(td)^{2l} \,
\gamma_{2l}}{(2l)!} + \sum_{l=\frac{m}{2}}^{\infty}
\frac{(td)^{2l} \, \gamma_{m}}{(2l)!} \nonumber \\
&& \hspace*{-0.5cm} = 1 + \sum_{l=1}^{\frac{m}{2}-1} \frac{(td)^{2l} \,
\bigl(\gamma_{2l} - \gamma_m \bigr)}{(2l)!} + \gamma_m \bigl(\cosh(td)-1 \bigr)
\label{eq: chain of equalities for the conditional expectation}
\end{eqnarray}
where the inequality above holds since $|\frac{\xi_k}{d}| \leq 1$ a.s., so that
$0 \leq \ldots \leq \gamma_m \leq \ldots \leq \gamma_4 \leq \gamma_2 \leq 1$, and the
last equality in \eqref{eq: chain of equalities for the conditional expectation} holds
since $\cosh(x) = \sum_{n=0}^{\infty} \frac{x^{2n}}{(2n)!}$ for every $x \in \reals$.
Therefore, from \eqref{eq: smoothing theorem},
\begin{equation}
\expectation \biggl[ \exp \biggl(t \sum_{k=1}^n \xi_k \biggr) \biggr] \leq
\left(1 + \sum_{l=1}^{\frac{m}{2}-1} \frac{(td)^{2l} \,
\bigl(\gamma_{2l} - \gamma_m \bigr)}{(2l)!} + \gamma_m \bigl[\cosh(td)-1 \bigr] \right)^n
\label{eq: 2nd inequality for the expected value of the sum under a symmetry assumption}
\end{equation}
for an arbitrary $t \in \reals$. The inequality then follows from
\eqref{eq: maximal inequality for submartingales}.
This completes the proof of
Theorem~\ref{theorem: second inequality for conditionally symmetric martingales}.

\subsection{A Proof of Theorem~\ref{theorem: improved inequality
for the exercise of Amir and Ofer's book w.r.t. conditionally
symmetric martingales}}
The proof of Theorem~\ref{theorem: improved inequality
for the exercise of Amir and Ofer's book w.r.t. conditionally
symmetric martingales} relies on the proof of the known result
in Theorem~\ref{theorem: from the book of Dembo and Zeitouni},
where the latter dates back to Freedman's paper (see
Theorem~1.6 in Freedman (1975), and also
Exercise~2.4.21(b) in Dembo and Zeitouni (1997)). The original proof
of Theorem~\ref{theorem: from the book of Dembo and Zeitouni}
(see Section~3 in Freedman (1975)) is modified in a way that
facilitates realizing how the bound can be improved for conditionally
symmetric martingales with bounded jumps. This improvement is
obtained via the refinement
of Bennett's inequality for conditionally symmetric distributions.

Without any loss of generality, let us assume that $d=1$
(otherwise, $\{X_k\}$ and $z$ are divided by $d$, and $\{Q_k\}$
and $r$ are divided by $d^2$; this normalization extends
the bound to the case of an arbitrary $d>0$).
Let $S_n \triangleq X_n - X_0$ for every $n \in \naturals_0$;
then $\{S_n, \mathcal{F}_n\}_{n \in \naturals_0}$ is a
martingale with $S_0=0$.

\begin{lemma}
Under the assumptions of Theorem~\ref{theorem: improved inequality
for the exercise of Amir and Ofer's book w.r.t. conditionally symmetric
martingales}, let
\begin{equation}
U_n \triangleq \exp(\lambda S_n - \theta Q_n), \quad \forall \, n \in
\{0, 1, \ldots \} \label{eq: U supermartingale}
\end{equation}
where $\lambda \geq 0$ and $\theta \geq \cosh(\lambda)-1 \triangleq \theta_{\min}(\lambda)$ are
arbitrary constants. Then, $\{U_n, \mathcal{F}_n\}_{n \in \naturals_0}$
is a supermartingale. \label{lemma: strengthened result for the supermartingale}
\end{lemma}
\begin{IEEEproof}
It is easy to verify that $U_n \in L^1(\Omega, \mathcal{F}_n, \pr)$
for $\lambda, \theta \geq 0$ (note that $S_n \leq n$ a.s.). It is required to show that
$\expectation[U_n | \mathcal{F}_{n-1}] \leq U_{n-1}$ holds a.s.
for every $n \in \naturals$, under the above
assumptions on $\lambda$ and $\theta$ in
\eqref{eq: U supermartingale}. We have
\begin{eqnarray}
&& \expectation[U_n | \mathcal{F}_{n-1}] \nonumber\\
&& \stackrel{\text{(a)}}{=} \exp(-\theta Q_n) \, \exp(\lambda
S_{n-1}) \, \expectation\bigl[ \exp(\lambda \xi_n) \, |
\, \mathcal{F}_{n-1}\bigr] \nonumber\\
&& \stackrel{\text{(b)}}{=} \exp(\lambda S_{n-1}) \,
\exp\bigl(-\theta(Q_{n-1}+\expectation[\xi_n^2 |
\mathcal{F}_{n-1}])\bigr) \; \expectation\bigl[\exp(\lambda \xi_n)
\, | \, \mathcal{F}_{n-1}\bigr] \nonumber \\
&& \stackrel{\text{(c)}}{=} U_{n-1} \left(\frac {\expectation\bigl[
\exp(\lambda \xi_n) \, | \, \mathcal{F}_{n-1}\bigr]}{\exp(\theta
\expectation[\xi_n^2 \, | \, \mathcal{F}_{n-1}])} \right)
\label{eq: chain of equalities for the conditional expectation of U}
\end{eqnarray}
where (a) follows from \eqref{eq: U supermartingale} and because $Q_n$ and $S_{n-1}$ are
$\mathcal{F}_{n-1}$-measurable and $S_n = S_{n-1} + \xi_n$, (b) follows from
\eqref{eq: Q}, and (c) follows from \eqref{eq: U supermartingale}.

By assumption $\xi_n = S_n - S_{n-1} \leq 1$ a.s., and
$\xi_n$ is conditionally symmetric
around zero, given $\mathcal{F}_{n-1}$, for every
$n \in \naturals$. By applying Corollary~\ref{corollary:
refined Bennett's inequality for symmetric distributions}
to the conditional expectation of $\exp(\lambda \xi_n)$
given $\mathcal{F}_{n-1}$, it follows from
\eqref{eq: chain of equalities for the conditional expectation of U} that
\begin{eqnarray}
\expectation[U_n | \mathcal{F}_{n-1}] \leq U_{n-1}
\left(\frac{1 + \expectation[\xi_n^2 \, | \, \mathcal{F}_{n-1}] \;
\bigl(\cosh(\lambda)-1\bigr)}{\exp\bigl(\theta \expectation[\xi_n^2 |
\mathcal{F}_{n-1}]\bigr)} \right). \label{eq: inequality for the
conditional expectation of U_n under the conditional symmetry assumption}
\end{eqnarray}
Let $\lambda \geq 0$. In order to ensure that $\{U_n, \mathcal{F}_n\}_{n \in \naturals_0}$ forms
a supermartingale, it is sufficient (based on \eqref{eq: inequality for the
conditional expectation of U_n under the conditional symmetry assumption})
that the following condition holds:
\begin{equation}
\frac{1+\alpha \bigl(\cosh(\lambda)-1\bigr)}{\exp(\theta \alpha)} \leq 1,
\quad \forall \, \alpha \geq 0.
\label{eq: condition for supermartingality under the conditional symmetry assumption}
\end{equation}
Calculus shows that, for $\lambda \geq 0$, the condition
in \eqref{eq: condition for supermartingality under the conditional symmetry assumption}
is satisfied if and only if
\begin{equation}
\theta \geq \cosh(\lambda)-1 \triangleq \theta_{\min}(\lambda).
\label{eq: new minimal theta}
\end{equation}
From \eqref{eq: inequality for the conditional expectation of U_n under
the conditional symmetry assumption},
$\{U_n, \mathcal{F}_n\}_{n \in {\naturals_0}}$
is a supermartingale if $\lambda \geq 0$ and $\theta \geq \theta_{\min}(\lambda)$.
This proves Lemma~\ref{lemma: strengthened result for the supermartingale}.
\end{IEEEproof}

Let $z, r > 0$, $\lambda \geq 0$, and
$\theta \geq \cosh(\lambda) - 1$. In the following, we
rely on Doob's sampling theorem.
To this end, let $M \in \naturals$, and define
two stopping times adapted to $\{\mathcal{F}_n\}$.
The first stopping time is $\alpha = 0$, and the
second stopping time $\beta$ is the
minimal value of $n \in \{0, \ldots, M\}$
(if any) such that $S_n \geq z$ and $Q_n \leq r$
(note that $S_n$ is $\mathcal{F}_n$-measurable
and $Q_n$ is $\mathcal{F}_{n-1}$-measurable, so
the event $\{\beta \leq n\}$
is $\mathcal{F}_n$-measurable); if such a value
of $n$ does not exist, let $\beta \triangleq M$. Hence
$\alpha \leq \beta$ are two bounded stopping times.
From Lemma~\ref{lemma: strengthened result for the supermartingale},
$\{U_n, \mathcal{F}_n\}_{n \in \naturals_0}$
is a supermartingale for the corresponding set of parameters
$\lambda$ and $\theta$, and from Doob's sampling theorem
\begin{equation}
\expectation[U_{\beta}] \leq \expectation[U_0] = 1
\label{eq: Doob's sampling theorem for supermartingales}
\end{equation}
($S_0 = Q_0 = 0$, so from \eqref{eq: U supermartingale}, $U_0 = 1$ a.s.).
Hence, this implies the following chain of inequalities:
\begin{eqnarray}
&& \pr(\exists \, n \leq M: \, S_n \geq z, Q_n \leq r) \nonumber\\
&& \stackrel{\text{(a)}}{=} \pr(S_{\beta} \geq z, Q_{\beta} \leq r)
\nonumber\\
&& \stackrel{\text{(b)}}{\leq} \pr(\lambda S_{\beta} - \theta Q_{\beta}
\geq \lambda z - \theta r)
\nonumber \\
&& \stackrel{\text{(c)}}{\leq} \frac{\expectation[U_{\beta}]}{\exp(\lambda z -
\theta r)} \nonumber \\
&& \stackrel{\text{(d)}}{\leq} \exp\bigl(-(\lambda z - \theta r)\bigr)
\label{eq: chain of inequalities for the tail probability of
martingales}
\end{eqnarray}
where equality~(a) follows from the definition of the stopping time
$\beta \in \{0, \ldots, M\}$, (b) holds since $\lambda, \theta \geq 0$,
(c) follows from Chernoff's bound and the definition
in \eqref{eq: U supermartingale}, and finally~(d) follows from
\eqref{eq: Doob's sampling theorem for supermartingales}. Since
\eqref{eq: chain of inequalities for the tail probability of
martingales} holds for every $M \in \naturals$, from
the continuity theorem for non-decreasing events and \eqref{eq: chain
of inequalities for the tail probability of martingales}
\begin{eqnarray}
&& \pr(\exists \, n \in \naturals: \, S_n \geq z, Q_n \leq r)
\nonumber\\[0.1cm]
&& = \lim_{M \rightarrow \infty} \pr(\exists \, n \leq M:
\, S_n \geq z, Q_n \leq r) \nonumber\\[0.1cm]
&& \leq \exp\bigl(-(\lambda z - \theta r)\bigr).
\label{eq: 2nd chain of inequalities for the tail probability
of martingales}
\end{eqnarray}
The choice of the non-negative parameter $\theta$ as the minimal value for which \eqref{eq:
2nd chain of inequalities for the tail probability of martingales} is
valid provides the tightest bound within this form. Hence, for a fixed $\lambda \geq 0$
and $\theta = \theta_{\min}(\lambda)$, the bound in
\eqref{eq: 2nd chain of inequalities for the tail probability of martingales}
gives the inequality
\begin{eqnarray*}
\pr(\exists \, n \in \naturals: \, S_n \geq z, Q_n \leq r)
\leq \exp\Bigl(-\bigl[\lambda z - r \, \theta_{\min}(\lambda)
\bigr]\Bigr), \quad \forall \, \lambda \geq 0.
\end{eqnarray*}
The optimized $\lambda$ is equal to $\lambda = \sinh^{-1} \Bigl(\frac{z}{r}\Bigr).$
Its substitution in \eqref{eq: new minimal theta} gives that
$\theta_{\min}(\lambda) = \sqrt{1+\frac{z^2}{r^2}}-1$,
and
\begin{equation}
\pr(\exists \, n \in \naturals: \, S_n \geq z, Q_n \leq r)
\leq \exp\left(-\frac{z^2}{2r} \cdot C\Bigl(\frac{z}{r}\Bigr) \right)
\label{eq: tightened inequality for d = 1 under conditional symmetry}
\end{equation}
with $C$ in \eqref{eq: C}.
Finally, the proof of Theorem~\ref{theorem:
improved inequality for the exercise of Amir and Ofer's book
w.r.t. conditionally symmetric martingales}
is completed by showing that the following equality holds:
\begin{eqnarray}
&& \hspace*{-1cm} A \triangleq \{\exists \, n \in \naturals:
\, S_n \geq z, Q_n \leq r\} \nonumber \\
&& \hspace*{-0.6cm}  = \{\exists \, n \in \naturals:
\, \max_{1 \leq k \leq n} S_k \geq z, Q_n \leq r\} \triangleq B.
\label{eq: equivalent events}
\end{eqnarray}
Clearly $A \subseteq B$. To show that $B \subseteq A$,
assume that event $B$ is satisfied. Then, there exists some
$n \in \naturals$ and $k \in \{1, \ldots, n\}$ such that $S_k \geq z$
and $Q_n \leq r$. Since the predictable quadratic variation process
$\{Q_n\}_{n \in \naturals_0}$ in \eqref{eq: Q}
is monotonic non-decreasing, this implies that
$S_k \geq z$ and $Q_k \leq r$; therefore, event $A$ is also satisfied
and $B \subseteq A$. The combination of
\eqref{eq: tightened inequality for d = 1 under conditional symmetry}
and \eqref{eq: equivalent events} completes the proof of
Theorem~\ref{theorem: improved inequality for the exercise of Amir
and Ofer's book w.r.t. conditionally symmetric martingales}.

\subsection{Proof of Corollary~\ref{corollary: concentration
inequality for conditionally symmetric sub and supermartingales}}

The proof of Corollary~\ref{corollary: concentration
inequality for conditionally symmetric sub and supermartingales}
is similar to the proof of
Theorem~\ref{theorem: a concentration inequality for a symmetric
conditional distribution}.
The only difference is that for a supermartingale,
$X_k - X_0 = \sum_{j=1}^k (X_j-X_{j-1}) \leq \sum_{j=1}^k
\eta_j$ a.s., where $\eta_j \triangleq X_j-\expectation[X_j \, | \,
\mathcal{F}_{j-1}]$ is $\mathcal{F}_j$-measurable. Hence
$\pr\Bigl(\max_{1 \leq k \leq n} X_k-X_0 \geq \alpha n\Bigr) \leq
\pr\biggl(\max_{1 \leq k \leq n} \sum_{j=1}^k \eta_j \geq
\alpha n\biggr)$ where a.s. $\eta_j \leq d$, $\expectation[\eta_j \,
| \, \mathcal{F}_{j-1}]=0$, and $\text{Var}(\eta_j \, | \,
\mathcal{F}_{j-1}) \leq \sigma^2$. The continuation
coincides with the proof of Theorem~\ref{theorem: a concentration
inequality for a symmetric conditional distribution} for the
martingale $\{\sum_{j=1}^k \eta_j, \mathcal{F}_k\}_{k \in \naturals_0}$
(starting from \eqref{eq: maximal inequality for submartingales} and
\eqref{eq: inequality for the expected
value of the sum under a symmetry assumption}).
The passage to submartingales is trivial.

\subsection{Proof of Corollary~\ref{corollary: adaptation of
Theorems 6 and 7 to supermartingales}}

Consider the martingale $\{Y_n, \mathcal{F}_n\}_{n \in \naturals_0}$
where $Y_n \triangleq \sum_{j=1}^n \eta_j$
with $\eta_j \triangleq X_j-\expectation[X_j \, | \,
\mathcal{F}_{j-1}]$, and $Y_0 \triangleq 0$.  Since
$Y_k - Y_{k-1} = \eta_k$ for every $k \in \naturals$,
the predictable quadratic variation process $\{Q_n\}_{n \in \naturals_0}$
which corresponds to the martingale $\{Y_n, \mathcal{F}_n\}_{n \in \naturals_0}$
is, from \eqref{eq: Q}, the same process as the
one which corresponds to the supermartingale
$\{X_n, \mathcal{F}_n\}_{n \in \naturals_0}$. Furthermore,
$X_k - X_0 \leq \sum_{j=1}^n \xi_j = Y_k - Y_0$
for every $k \in \naturals$.
Hence, for every $z, r > 0$,
\begin{eqnarray*}
&& \pr\Bigl( \exists \, n \in \naturals: \,
\max_{1 \leq k \leq n} (X_k - X_0)
\geq z, \, Q_n \leq r \Bigr)
\leq \pr\Bigl( \exists \, n \in \naturals: \,
\max_{1 \leq k \leq n} (Y_k - Y_0)
\geq z, \, Q_n \leq r \Bigr).
\end{eqnarray*}
If $\{X_n, \mathcal{F}_n\}_{n \in \naturals_0}$ is a conditionally symmetric
supermartingale, then $\{Y_n, \mathcal{F}_n\}_{n \in \naturals_0}$ is
a conditionally symmetric martingale. Then, applying
Theorem~\ref{theorem: improved inequality for the exercise of Amir
and Ofer's book w.r.t. conditionally symmetric martingales}
to the martingale $\{Y_n, \mathcal{F}_n\}_{n \in \naturals_0}$
gives the improved bound in \eqref{eq: improved inequality for the
exercise of Amir and Ofer's book w.r.t. conditionally symmetric martingales}
and \eqref{eq: C}.

\section*{Acknowledgment}
The anonymous reviewer and the Associate Editor are acknowledged for their 
detailed comments that improved the presentation of this paper. Suggestions 
made by Ofer Zeitouni were also helpful in improving the presentation of the results.

\section*{References}
{\footnotesize \hspace*{-0.45cm}
Alon, N., Spencer, J.H., 2008. The Probabilistic Method, third ed. Wiley
Series in Discrete Mathematics and Optimization.\\
Azuma, K., 1967. Weighted sums of certain dependent random variables.
The Tohoku Mathematical Journal~19, 357--367.\\
Bennett, G., 1962. Probability inequalities for the sum of independent
random variables. Journal of the American Statistical Association~57, 33--45.\\[0.1cm]
Burkholder, D.L., 1991. Explorations in martingale theory and its
applications. Ecole d'Et\'{e} de Probabilit\'{e}s de Saint-Flour~XIX--1989,
Lecture Notes in Mathematics, vol.~1464. Springer-Verlag, pp.~1--66.\\
Chung, F., Lu, L., 2006. Concentration inequalities and martingale
inequalities: a survey. Internet Mathematics~3, 79--127.\\
De la Pe\~{n}a, V.H., 1999. A general class of exponential inequalities for
martingales and ratios. Annals of Probability~27, 537--564.\\
De la Pe\~{n}a, V.H., Klass, M.J., Lai, T.L., 2004. Self-normalized processes:
exponential inequalities, moment bounds and iterated logarithm laws.
Annals of Probability~32, 1902--1933.\\
Dembo, A., Zeitouni, O., 1997. Large Deviations Techniques and
Applications, second ed. Springer.\\
Denuit, M., Dhaene, J., Goovaerts, M.J., Kaas, R., 2005. Actuarial
Theory for Dependent Risks, John Wiley \& Sons Inc., NY.\\
Dzhaparide, K., van Zanten, J.H., 2001. On Bernstein-type inequalities
for martingales. Stochastic Processes and their Applications~93, 109--117.\\
Freedman, D., 1975. On tail probabilities for martingales.
Annals of Probability 3, 100--118.\\
Grama, I., Haeusler, E., 2000. Large deviations for martingales via
Cramer's method. Stochastic Processes and their Applications~85, 279--293.\\
Hoeffding, W., 1963. Probability inequalities for sums of bounded
random variables. Journal of the American Statistical
Association~58, 13--30.\\
McDiarmid, C., 1989. On the method of bounded differences.
In: Surveys in Combinatorics, vol.~141. Cambridge
University Press, Cambridge, pp.~148--188.\\
McDiarmid, C., 1998. Concentration. In: Probabilistic Methods for
Algorithmic Discrete Mathematics. Springer, pp.~195--248.\\
Os\c{e}kowski, A., 2010a.
Weak type inequalities for conditionally symmetric martingales.
Statistics and Probability Letters~80, 2009--2013. \\
Os\c{e}kowski, A., 2010b. Sharp ratio inequalities for a conditionally symmetric
martingale. Bulletin of the Polish Academy of Sciences Mathematics~58,
65--77. \\
Pinelis, I., 1994. Optimum bounds for the distributions of martingales
in Banach spaces. Annals of Probability~22, 1679--1706.\\
Steiger, W.L., 1969. A best possible Kolmogoroff-type inequality for
martingales and a characteristic property. Annals of
Mathematical Statistics~40, 764--769.\\
Wang, G., 1991. Sharp maximal inequalities for conditionally symmetric
martingales and Brownian motion. Proceedings of the American
Mathematical Society~112, 579--586.
}
\end{document}